\documentclass[a4paper,11pt,reqno]{amsart}
\usepackage{a4wide,amssymb,latexsym}
\usepackage{eucal}

\usepackage{amsmath}
\usepackage{amsthm}
\usepackage{amstext}
\usepackage{amsfonts}
\usepackage{graphicx}
\usepackage{dsfont}

\usepackage[usenames,dvipsnames]{pstricks}
\usepackage{epsfig}
\usepackage{pst-grad} 
\usepackage{pst-plot} 
\numberwithin{equation}{section}

\newcommand{\R}{\mathbb R}
\newcommand{\C}{\mathbb C}

\newcommand{\Z}{\mathbb Z}

\newcommand{\cH}{\mathcal H}
\newcommand{\cI}{\mathcal I}
\newcommand{\cN}{\mathcal N}

\newcommand{\cT}{\mathcal T}
\newcommand{\be}{\begin{equation}}
\newcommand{\ee}{\end{equation}}
\newcommand{\ba}{\begin{eqnarray}}
\newcommand{\ea}{\end{eqnarray}}

\def\R{\mathbb{R}}

\newtheorem{prop}{Proposition}[section]
\newtheorem{thm}{Theorem}

\newtheorem{lem}{Lemma}
\newtheorem{Rk}{Remark}

\newtheorem{Exam}{Example}

\begin{document}

\title{Finite-time stabilization of a network of strings}

\author{Fatiha Alabau-Boussouira}
\address{Institut Elie Cartan de Lorraine, UMR 7502, Universit\'e de Lorraine, 
F-57045, Metz, France}
\email{alabau@univ-metz.fr}

\author{Vincent Perrollaz}
\address{Laboratoire de Math\'ematiques et Physique Th\'eorique, 
Universit\'e de Tours,
UFR Sciences et Techniques,
Parc de Grandmont,
37200 Tours, France}
\email{Vincent.Perrollaz@lmpt.univ-tours.fr}

\author{Lionel Rosier}
\address{Centre Automatique et Syst\`emes, MINES ParisTech, PSL Research University, 60 Boulevard Saint-Michel, 75272 Paris Cedex 06, France}
\email{Lionel.Rosier@mines-paristech.fr}


\begin{abstract} 
We investigate the finite-time stabilization of a tree-shaped network of strings. Transparent boundary conditions are applied  at all the external nodes.
At any internal node, in addition to the usual continuity conditions, a modified Kirchhoff law incorporating a  damping term $\alpha u_t$ with a
coefficient $\alpha$ that may depend on the node is considered. 
We show that for a convenient choice of the sequence of coefficients $\alpha$, any solution of the wave equation on the network becomes constant 
after a finite time. The condition on the coefficients proves to be sharp at least for a star-shaped tree. Similar results are derived when we replace the transparent
boundary condition by the Dirichlet (resp. Neumann) boundary condition at one external node.   
\end{abstract}

\maketitle
\section{Introduction}
Solutions of certain ODE $\dot x = f(x)$ may reach the equilibrium state in finite time. This phenomenon, when combined with the stability,  was termed
{\em finite-time} stability in \cite{BB,haimo}. 

A {\em finite-time stabilizer} is a feedback control for which the closed-loop system is finite-time stable around some equilibrium state. 
In some sense, it satisfies a controllability  objective with a control in feedback form. On the other hand, a finite-time stabilizer may 
be seen as an exponential stabilizer yielding an arbitrarily large decay rate for the solutions to the closed-loop system. Indeed, any solution
of the closed-loop system can be  estimated as
\[
||x(t)|| \le h(||x_0||){\bf 1}_{[0,T]}(t) \le h(||x_0||)e^{-\lambda (t-T)}
\]
where $h(\delta )\to 0$ as $\delta \to 0$, and $\lambda>0$ is arbitrarily large. 
This explains why some efforts were made in the last decade to construct finite-time stabilizers for controllable systems, including  the linear ones. 
See  \cite{MP} for some recent developments
and up-to-date references, and \cite{BR} for some connections with Lyapunov theory. 

To the best knowledge of the authors, the analysis of the finite-time stabilization of PDE is not developed yet. However, since \cite{majda}, it is well-known 
that solutions of the wave equation on certain bounded domains may disappear when using {\em transparent} boundary conditions. For instance, the solution of 
the 1-D wave equation 
\ba
u_{tt} - c^2u_{xx}=0,&&  \text{in }  (0,T)\times (0,L), \label{Int1}\\
c u_x(L,t)= - u_t(L,t),&& \text{in }  (0,T),  \label{Int2}\\
c u_x(0,t)= u_t (0,t),&& \text{in }  (0,T),  \label{Int3}\\
(u(0),u_t(0))=(u^0,u^1), &&\text{in } (0,L),\label{Int4}
\ea
is finite-time stable in the space $\{ (u,v)\in H^1(0,L)\times L^2(0,L); \ c\big( u(0)+u(L)\big) +\int_0^L v(x)dx=0\}$, with $T=L/c$ as extinction time (see e.g. \cite[Theorem 0.5]{komornik}
for the details.) The condition \eqref{Int2} is ``transparent'' in the sense that a wave $u(x,t)=f(x-ct)$ traveling to the right satisfies \eqref{Int2} and leaves the domain
at $x=L$ without generating any reflected wave. Note that the solution issued from any state $(u^0,u^1)\in H^1(0,L)\times L^2(0,L)$ is not necessarily vanishing, but constant, for  
$t\ge L/c$. 
Note also that if we replace \eqref{Int3} by the boundary condition $u(0,t)=0$ (or $u_x(0,t)=0$),  then a finite-time
extinction still occurs (despite the fact that waves bounce at $x=0$) with an extinction time $T=2L/c$.  We refer to \cite{CZ} for the analysis of the finite-time 
extinction property for a nonhomogeneous string with a viscous damping at one extremity, to \cite{gugat} for the finite-time stabilization of a string with a moving boundary, 
to \cite{PR-IFAC} (resp. \cite{PR}) for the finite-time stabilization 
of a system of conservation laws on an interval (resp. on a tree-shaped network).

The finite-time stability of \eqref{Int1}-\eqref{Int4} is easily established when writing \eqref{Int1} as a system of two transport equations
\begin{eqnarray*}
&&d_t + cd_x =0,\label{Int11} \\
&&s_t-  cs_x =0.\label{Int12}
\end{eqnarray*} 
where $d:=u_t-cu_x$ and $s:=u_t+cu_x$ stand for the Riemann invariants for the wave equation written as a first order hyperbolic system.  
The boundary conditions \eqref{Int2} and  \eqref{Int3} yield $d(0,t)=s(L,t)=0$ (and hence $d(.,t)=s(.,t)=0$ for $t\ge L/c$), while the boundary conditions
\eqref{Int2} and $u(0,t)=0$ yield $s(L,t)=0$ and $d(0,t)= - s(0,t)$ (and hence $s(.,t)=0$ for $t\ge L/c$ and $d(.,t)=0$ for $t\ge 2L/c$).

The stabilization of networks of strings has been considered in e.g. \cite{AJ,AJK,DZbook,GDL,GS,VZ,XLL}. In \cite{GS}, the authors considered a star of vibrating strings,
and derived the finite time stability (resp. the exponential stability) when transparent boundary conditions are applied at all external nodes (resp. at all external nodes but one, which is changing as times proceeds). For a more general network, we guess that the finite time stability cannot hold without the 
introduction of additional feedback controls at the internal nodes. 
Indeed, it is proved here that for a bone-shaped tree, if the feedback controls are applied only at the external nodes, then the finite time stability fails. 

The aim of this paper is to investigate the {\em finite-time} stabilization of a {\em tree-shaped} network of strings. At each internal node $n$ connecting $k$ edges,
we  assume that the usual continuity condition hold
\be
\label{Int21}
u_i(n,t)=u_j(n,t), \qquad \forall i\ne j,
\ee 
while the usual Kirchhoff law is modified by incorporating a {\em damping term} inside:
\be
\label{Int22}
\sum_{i} c_iu_{i,x}(n,t) = -\alpha (n) u_t (n,t).
\ee 
In \eqref{Int22}, the sum is over the indices $i$ of the edges having $n$ as one end,  
$\alpha (n)  \in\R$ is a coefficient depending on the node $n$,  and 
we have set $u(n,t):=u_i(n,t)$ (for any $i$) and taken $n$ as the origin of each edge to define the derivative 
along the space variable.  The case $\alpha =0$ corresponds to the usual (conservative) Kirchhoff law. 

Note that we can assume without loss of generality that the length of each edge is one, by scaling the variable $x$ and the coefficient $c_i$ along each edge.

Even if the finite-time stabilization of  $2\times 2$ hyperbolic systems on tree-shaped networks  was already considered in \cite{PR} (and applied to the regulation of water flows 
in networks of canals, with $k-1$ controls at any node connecting $k$ canals), the novelty (and difficulty) here comes from the fact that only {\em one} control 
is applied at each internal node. The present work can be seen as a first step in the understanding of the finite-time stabilization of systems of conservation laws with a few controls. 
 
A natural guess is that the finite-time stability cannot hold if one can find in the tree a pair of adjacent nodes that are free of any control, because of the (partial but standing)
bounces of waves at these nodes.  This conjecture will be demonstrated here for a star-shaped tree and a bone-shaped tree.

Actually, we shall prove that the finite-time stabilization can be achieved for a very particular choice of the coefficient $\alpha$ at each internal node. 
One of the main results proved in this paper is the following
\begin{thm}
\label{thm-intro}
Consider any tree-shaped network of strings, with transparent boundary conditions at the external nodes,  continuity conditions and the modified Kirchhoff law
at the internal nodes. If at each internal node $n$ connecting $k$ edges we have $\alpha (n)=k-2$, then each solution of the wave equation on the network 
becomes constant after some finite time. 
\end{thm} 
Similar results will be obtained when replacing at one given external node the transparent boundary condition by the homogeneous Dirichlet (resp. Neumann) boundary condition.
We shall also see that the condition about $\alpha$ is {\em sharp} for a star-shaped tree by explicit computation of the discrete spectrum. The same approach gives for
a bone-shaped tree a necessary and sufficient condition for the finite time stability, which differs slightly from those stated in Theorem \ref{thm-intro}.

The paper is outlined as follows. In Section 2, we provide a sharp condition on the coefficients $\alpha (n)$ for the system to be well-posed. It is obtained
by expressing the conditions \eqref{Int21}-\eqref{Int22}  at the internal nodes in terms of the Riemann invariants. In Section 3, we prove the finite-time stability results
when the coefficients $\alpha$ are chosen as in  Theorem \ref{thm-intro}. We discuss in  
Section 4 the necessity of that condition by considering  tree-shaped networks and bone-shaped networks. 

\section{Well-posedness}
We introduce some notations inspired by  \cite{DZcras}. Let $\mathcal T$ be a tree, whose {\em vertices} (or {\em nodes})
are numbered by the index $n\in {\mathcal N}  =\{ 0,...,N\}$, and whose {\em edges} are numbered by the index $i\in {\mathcal I} =\{ 1,...,N\}$. We choose 
a simple vertex (i.e. an external node), called the {\em root}  of $\mathcal T$ and denoted by $\mathcal R$, and which corresponds to the index $n=0$.
The edge containing $\mathcal R$ has $i=1$ as index, and its other endpoint has for index $n=1$.  
We choose an orientation of the edges in the tree such that $\mathcal R$ is the ``first'' encountered vertex.  The {\em depth} $d$ of the tree is the number of 
generations ($d=1$ for a tree reduced to a single edge, $d=2$ for a star-shaped tree, etc.) 
Once the orientation of the tree is chosen, each point of the $i$-th edge (of length 1) is identified with
a real number $x\in [0,1]$. The points  $x=0$ and $x=1$ are termed  the {\em initial point} and the {\em final point} of the $i$-th edge, respectively. 
Renumbering the edges if needed, we can assume that the edge of index $i$ has as final point the vertex with the (same) index $n=i$ for all $i\in {\mathcal I}$.  
(See Figure \ref{fig1}.)
\begin{figure}[http]
\begin{center}
\includegraphics[scale=0.5]{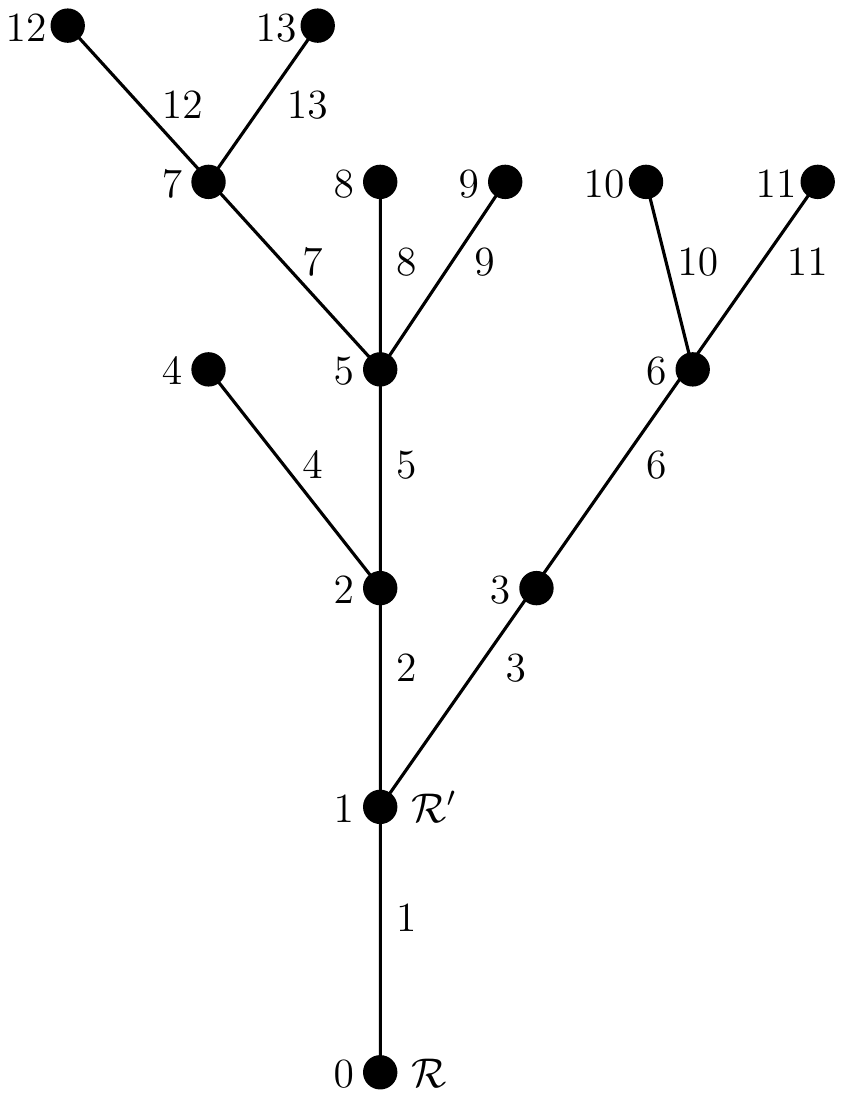}
\end{center}
\caption{ A tree with $14$ nodes, a depth equal to $5$, with simple nodes $\mathcal{N}_S=\{0,4,8,9,10,11,12,13\}$ and multiple nodes $\mathcal{N}_M=\{ 1 , 2 , 3, 5 , 6, 7 \}$.}
\label{fig1}
\end{figure}
The set of  indices of simple and multiple nodes are denoted by ${\mathcal N}_S$ and ${\mathcal N}_M$, respectively.

For $n\in {\mathcal N}_M$ we denote by ${\mathcal I}_n$ the set of indices of those edges having the vertex of index $n$ as initial point.
As we consider a network of strings whose constants $c_i$ 
may vary from one edge to another one, the case $\# ({\mathcal I}_n)=1$ (one child) is possible. The number of edges having the vertex of index $n$ 
as one as their extremities is 
\[
k_n:= \# ({\mathcal I}_n )  + 1\ge 2.
\]
We consider the following system
\ba
u_{i,tt} - c_i^2 u_{i,xx} &=& 0, \qquad t>0, \ 0<x<1, \   i\in {\mathcal I} \label{A1} \\
(u_i(.,0),u_{i,t}(.,0)) &=& (u_{i}^0,u_{i}^1), \qquad i\in {\mathcal I} \label{A2}
\ea
with the following boundary conditions
\ba
c_nu_{n,x}(1,t) &=& -u_{n,t}(1,t), \qquad t>0, \ n\in {\mathcal N} _S\setminus \{ 0\} , \label{A3} \\
\sum_{i\in {\mathcal I}_n } c_i u_{i,x}(0,t) - c_n u_{n,x} (1,t) &=& -\alpha _n u_{n,t}  (1,t) ,\qquad   t>0, \ n\in {\mathcal N}_M, \label{A4}\\
u_i(0,t) &=& u_{n} (1,t),\qquad t>0,\  n\in {\mathcal N}_M, \ i\in {\mathcal I}_n, \label{A5}    
\ea
where the sequence $(\alpha _n)_{n\in {\mathcal N}_M}$ is still to be defined. For the boundary condition at the root $\mathcal R$, we 
shall consider one of the following conditions
\ba
\label{A6a}
u_1(0,t) &=& 0, \qquad t>0 \qquad \text{\rm (Dirichlet boundary condition)};\\
\label{A6b}
u_{1,x}(0,t)&=& 0, \qquad t>0 \qquad \text{\rm (Neumann boundary condition)}; \\
\label{A6c}
c_1u_{1,x}(0,t) &=& u_{1,t}(0,t), \qquad t>0 \qquad \text{\rm (Transparent boundary condition).}
\ea
Let 
\[
{\mathcal H} = \big\{ (u_i,v_i)_{i\in \mathcal I} \in \prod_{i\in \mathcal I} [H^1(0,1)\times L^2(0,1)];\ 
u_i(0)=u_n(1) \ \forall n\in {\mathcal N}_M, \ \forall i\in {\mathcal I}_n   
\big\} 
\]
and ${\mathcal H}_0 =\big\{ (u_i , v_i)_{i\in \mathcal I} \in {\mathcal H} ;\ u_1 (0)=0\}$.

Replacing $u_{i,t}$ by $v_i$ and dropping the variable $t$, conditions \eqref{A3} - \eqref{A6c} may be rewritten respectively as
\ba
c_nu_{n,x}(1)&=&-v_n(1), \qquad  n\in {\mathcal N} _S\setminus \{ 0\} , \label{AA3} \\
\sum_{i\in {\mathcal I}_n } c_i u_{i,x}(0) - c_n u_{n,x} (1) &=& -\alpha _n v_n(1) ,\qquad  n\in {\mathcal N}_M, \label{AA4}\\
u_i(0) &=& u_n (1),\qquad  n\in {\mathcal N}_M, \ i\in {\mathcal I}_n, \label{AA5}   \\ 
u_1(0)&=& 0, \label{AA6a}\\
u_{1,x}(0)&=&0, \label{AA6b}\\
c_1u_{1,x}(0)&=&v_1(0). \label{AA6c}
\ea
If $t\in \R ^+ \to (u_i,v_i)_{i\in \mathcal I}\in {\mathcal D}(A_T)$ is continuous, using $v_i=u_{i,t}$, \eqref{AA5} and \eqref{AA6a} we obtain  
\ba
v_i(0) &=& v_n (1),\qquad  n\in {\mathcal N}_M, \ i\in {\mathcal I}_n, \label{AAA5}   \\ 
v_1(0)&=& 0. \label{AAA6a}
\ea
Introduce the operator $A_D$, $A_N$ and $A_T$ defined as
\begin{eqnarray*}
&&A_D((u_i,v_i)_{i\in \mathcal I})=(v_i,c_i^2u_{i,xx})_{i\in \mathcal I}, \\
&&A_N((u_i,v_i)_{i\in \mathcal I})=(v_i,c_i^2u_{i,xx})_{i\in \mathcal I}, \\
&&A_T((u_i,v_i)_{i\in \mathcal I})=(v_i,c_i^2u_{i,xx})_{i\in \mathcal I}, 
\end{eqnarray*}
with respective domains
\begin{eqnarray*}
&&{ \mathcal D} (A_D)=
\{ (u_i,v_i)_{i\in \mathcal I}\in  \prod_{i\in \mathcal I} [H^2(0,1)\times H^1(0,1)];\  \eqref{AA3}-\eqref{AA5},\  \eqref{AA6a} \textrm{ and } 
\eqref{AAA5}-\eqref{AAA6a} \textrm{ hold} \} \\
&&\qquad\qquad  \subset {\mathcal H}_0, \\
&&{ \mathcal D} (A_N)=
\{ (u_i,v_i)_{i\in \mathcal I}\in  \prod_{i\in \mathcal I} [H^2(0,1)\times H^1(0,1)];\  \eqref{AA3}-\eqref{AA5}, \ \eqref{AA6b}   \textrm{ and }\eqref{AAA5}
\textrm{ hold} \}\subset {\mathcal H}, \\
&&{ \mathcal D} (A_T)=
\{ (u_i,v_i)_{i\in \mathcal I}\in  \prod_{i\in \mathcal I} [H^2(0,1)\times H^1(0,1)];\  \eqref{AA3}-\eqref{AA5}, \ \eqref{AA6c}   \textrm{ and }\eqref{AAA5}
\textrm{ hold} \}\subset{\mathcal H}.
\end{eqnarray*}

The main result in this section is  concerned with the well-posedness of system \eqref{A1}-\eqref{A5} and \eqref{A6a} (or \eqref{A6b}, or \eqref{A6c}). 
\begin{thm}
\label{thm1} 
Let $\mathcal T$ be a tree and let $(\alpha _n)_{n\in {\mathcal N}_M}$ be a given family of real numbers. 
Then $A_T$ generates a strongly continuous semigroup of operators on $\mathcal H$ if, and only if, 
\be
\label{cond-WP}
\alpha _n \ne k_n \qquad \forall n\in {\mathcal N}_M.
\ee
The same conclusion holds for $A_N$ on $\mathcal H$ (resp. for $A_D$ on $\mathcal H_0$). 
\end{thm}
\begin{proof}
We sketch the proof only for $A_T$. We need a preliminary result about the Riemann invariants
around an internal node. Consider any internal node connecting edges whose indices range over $\{1,...,k\}$ (to simplify the notations).
Consider any solution of \eqref{A1} satisfying 
\ba
&&u_1(1,t)=u_2(0,t)=\cdots = u_k(0,t)\label{AB1}\\
&&c_2u_{2,x}(0,t)+\cdots + c_k u_{k,x}(0,t)-c_1 u_{1,x}(1,t) = -\alpha u_{1,t}(1,t)  \label{AB2}
\ea
Introduce the Riemann invariants 
\ba
d_i(x,t)&:=&u_{i,t}(x,t)-c_iu_{i,x}(x,t),\label{Ri1}\\
s_i(x,t)&:=&u_{i,t}(x,t)+c_iu_{i,x}(x,t)\label{Ri2}
\ea 
for all $i\in \mathcal I$. Then the following result holds.
\begin{lem} 
\label{lem1}
\begin{enumerate}
\item If $\alpha \ne k$, then $s_1(1,t),d_2(0,t),...,d_k(0,t)$ can be expressed in a unique way as functions of 
$d_1(1,t),s_2(0,t),...,s_k(0,t)$.  In particular, if $\alpha =k-2$, we obtain
\be
s_1(1,t)=\sum_{i=2}^k s_i(0,t). \label{cle}
\ee
\item If $\alpha = k$, then the existence of a solution to \eqref{A1} and \eqref{AB1}-\eqref{AB2}  implies 
\begin{equation}
\label{CC}
d_1(1,t) +  \sum_{i=2}^{k}s_i(0,t) =0.
\end{equation}
This imposes that the initial condition $(u_i^0,v_i^0)_{i\in \mathcal I }$ satisfies the compatibility condition  
\be
(1-\alpha ) v_1^0(1) + \sum_{i=2}^k v_i^0(0)=0.
\label{iccc}
\ee
\end{enumerate}
\end{lem} 
{\em Proof of Lemma \ref{lem1}.}  Using Riemann invariants, we see that
\eqref{A1} and \eqref{AB1}-\eqref{AB2} are transformed into 
\ba
&&d_{i,t} +c_id_{i,x}=0,\qquad i=1,...,k,\label{C2} \\
&&s_{i,t} - c_is_{i,x}=0,\qquad i=1,...,k, \label{C1} \\
&&s_1(1,t)+d_1(1,t) = s_2(0,t) + d_2(0,t) = \cdots = s_k(0,t) + d_k(0,t),\label{C3} \\
&&\sum_{i=2}^k [ s_i(0,t)-d_i(0,t)] - (s_1(1,t) - d_1(1,t)) = -\alpha (s_1(1,t) + d_1(1,t))  \label{C4}
\ea 
To simplify the notations, we write $s_1$ for $s_1(1,t)$, $s_2$ for $s_2(0,t)$, etc. Then \eqref{C3}-\eqref{C4} can be written
\ba
&&s_1+d_1= d_i+s_i, \qquad i=2,...,k,\label{S1}\\
&&(1-\alpha )s_1+d_2 + \cdots + d_k = (1+\alpha ) d_1 + s_2 + \cdots + s_k \label{S2}
\ea
We readily infer from \eqref{S1} that 
\ba
s_1-d_2&=&-d_1+s_2, \label{S3}\\
d_2-d_3&=& -s_2+s_3,\label{S3bis}\\
&\vdots& \nonumber \\
d_{k-1}-d_k &=& -s_{k-1}+s_k.\label{S3ter}
\ea
Adding the $k-1$ equations in \eqref{S1} results in 
\begin{equation*}\label{Totor1}
(k-1)s_1 - \sum_{i=2}^{k} d_i= (1-k)d_1 + \sum_{i=2}^k s_i
\end{equation*}
Subtracting this last equation from \eqref{S2}, we obtain
\begin{equation*}
2\sum_{i=2}^{k} d_i= (k+\alpha)d_1 + (k+\alpha -2) s_1=
2d_1 + (k+\alpha -2) (d_1+s_1)
\end{equation*}
Combined to the relation $d_1+s_1=d_k+s_k$, this yields
\begin{equation*}
\sum_{i=2}^{k} d_i= d_1 + (\frac{k+\alpha}{2} -1)(d_k+s_k).
\end{equation*}
Using this relation in \eqref{S2} together with the relation $s_1=d_k+s_k -d_1$, we obtain
\begin{equation}\label{S4}
(k-\alpha)d_k= 2d_1 + 2\sum_{i=2}^{k-1} s_i + (\alpha -k+2) s_k.
\end{equation}
Thus, if $\alpha \ne k$, we infer from \eqref{S3}-\eqref{S4} that 
$s_1(1,t),d_2(0,t),...,d_k(0,t)$ can be expressed in a unique way as functions of 
$d_1(1,t),s_2(0,t),..., s_k(0,t)$.  In particular, if $\alpha =k-2$, then \eqref{S4} becomes
\be
d_k=d_1+\sum_{i=2}^{k-1} s_i. \label{S5}
\ee
Adding \eqref{S3},\eqref{S3bis},...,\eqref{S3ter} and \eqref{S5} yields \eqref{cle}.
Finally, if $\alpha =k$, then \eqref{S5} reads
\[
d_1(1,t) + \sum_{i=2}^k s_i(0,t) = 0.
\] 
Letting $t=0$ yields \eqref{CC}. Replacing $s_i$ and $d_i$ by their expressions in terms of $u_i$ and $v_i$ and using \eqref{AA4}, we obtain
\eqref{iccc}.\qed

Let us proceed to the proof of Theorem \ref{thm1}.  If \eqref{cond-WP} is not satisfied, picking some initial data $(u_i^0,v_i^0)_{i\in\mathcal I} \in {\mathcal D} (A_T)$
that does not satisfies \eqref{iccc} around an internal node for which \eqref{cond-WP} fails, we infer from Lemma \ref{lem1} that system \eqref{A1}-\eqref{A5} and \eqref{A6c}
does not admit any solution  $(u_i,v_i)_{i\in\mathcal I} \in C(\R ^+;{\mathcal D} (A_T))$. This shows $A_T$  is not the generator of a continuous semigroup on 
$\mathcal H$. Conversely, assume that  \eqref{cond-WP} is satisfied. We aim to construct by a fixed-point procedure a solution to \eqref{A1}-\eqref{A5} and \eqref{A6c}. 
Pick any $U^0=(u_i^0,v_i^0)_{i\in \mathcal I}\in \mathcal H$ and any $T>0$. Set 
\[
d_i^0:=v_i^0-c_i u_{i,x}^0, \qquad s_i^0:=v_i^0+c_i u_{i,x}^0, \qquad i=1,...,N.
\]
Pick a number $\rho \in (0,1)$. 
We introduce the Hilbert space $\mathcal E = L^2_{\rho ^tdt}(0,T)^N$ endowed with the norm 
\[
||(x_1,x_2,....,x_N)||^2_{\mathcal E} = \sum_{i=1}^N \int_0^T |x_i(t)|^2 \rho ^t dt.
\]
$X(t)$ stands for the vector $(...,d_n(1,t),s_{n+1}(0,t),...,s_{n+k_n-1}(0,t),...)$ where $n$ ranges over ${\mathcal N}_M$.
Let 
\[
{\mathcal E}_0:=\{(x_1,...,x_N)\in {\mathcal E};\  \ x_n(t)=0\ \forall t\ge c_n^{-1},\  \forall n\in {\mathcal N}_S\}.
\]
We define a map $P:X=(x_1,...,x_N)\in {\mathcal E}_0\to \tilde X=(\tilde x_1,...,\tilde x_N)\in {\mathcal E}_0$ as follows. 
Pick any $n\in {\mathcal N}_M$. By Lemma \ref{lem1}, there exists a 
matrix $A_n\in \R ^{k_n\times k_n}$ such that the Riemann invariants associated with the solution of \eqref{A1}-\eqref{A5} and \eqref{A6c} satisfy
\[
\left(\begin{array}{c}
s_n(1,t)\\
d_{n+1}(0,t)\\
\vdots\\
d_{n+k_n-1}(0,t)
\end{array} \right) 
=A_n 
\left(\begin{array}{c}
d_n(1,t)\\
s_{n+1}(0,t)\\
\vdots\\
s_{n+k_n-1}(0,t)
\end{array} \right). 
\]
Then, we set 
\[
\left(\begin{array}{c}
s_n(1,t)\\
d_{n+1}(0,t)\\
\vdots\\
d_{n+k_n-1}(0,t)
\end{array} \right) 
:=A_n 
\left(\begin{array}{c}
x_n(t)\\
x_{n+1}(t)\\
\vdots\\
x_{n+k_n-1}(t)
\end{array} \right). 
\]
Next, solving \eqref{C2}-\eqref{C1}, we set 
\be
\label{K1}
s_n(x,t) = \left\{
\begin{array}{ll}
s_n^0(x+c_nt) \quad &\textrm{ if } 0<x+c_nt<1,\\
s_n(1,t+c_n^{-1}(x-1))\quad &\textrm{ if } x+c_nt>1, 
\end{array}
\right.
\ee
and for $k=n+1,...,n+k_n-1$  
\be
\label{K2}
d_k(x,t) = \left\{
\begin{array}{ll}
d_k^0(x-c_kt) \quad &\textrm{ if } 0<x-c_kt<1,\\
d_k(0,t-c_k^{-1}x)\quad &\textrm{ if } x-c_kt<0. 
\end{array}
\right.
\ee
Similarly, we set 
\be
\label{K3}
d_n(x,t) = \left\{
\begin{array}{ll}
d_n^0(x-c_nt) \quad &\textrm{ if } 0<x-c_nt<1,\\
x_n(t+c_n^{-1}(1-x))\quad &\textrm{ if } x-c_nt<0, 
\end{array}
\right.
\ee
and for $k=n+1,...,n+k_n-1$  
\be
\label{K4}
s_k(x,t) = \left\{
\begin{array}{ll}
s_k^0(x+c_kt) \quad &\textrm{ if } 0<x+c_kt<1,\\
x_k(t+c_k^{-1}x)\quad &\textrm{ if } x+c_kt>1. 
\end{array}
\right.
\ee
Finally, we set
\[
\left(\begin{array}{c}
\tilde x_n(t)\\
\tilde x_{n+1}(t)\\
\vdots\\
\tilde x_{n+k_n-1}(t)
\end{array} \right) 
:=
\left(\begin{array}{c}
s_n(0,t)\\
d_{n+1}(1,t)\\
\vdots\\
d_{n+k_n-1}(1,t)
\end{array} \right). 
\]
Then it can be seen that $P$ is a map from ${\mathcal E}_0$ into itself.  Let us check that it is a contraction for $\rho$ small enough. 
Let $X^1=(x_1^1,...,x_N^1)$ and $X^2=(x_1^2,...,x_N^2)$ be given in $\mathcal E_0$. In what follows, $c$ denotes a constant that may vary from line to line. 
Then we have 
\ba
||P(X_1) -P(X_2)||^2_{\mathcal E}
&\le& c  \sum_{i=1}^N \int_{c_i^{-1} }^T |x_i^1(t-c_i^{-1})-x_i^2(t-c_i^{-1})|^2\rho^tdt\\
&\le& c(\max_{i\in \mathcal I} \rho ^{c_i^{-1}}) ||X^1-X^2||^2_{\mathcal E}. \ea
This proves that $P$ is a contraction in ${\mathcal E}_0$ for $\rho>0$ small enough. It follows from the contraction principle that $P$ has a (unique)
fixed-point in $\mathcal E_0$.  It is then easy to check that the Riemann invariants $d_i,s_i$, $1\le i\le N$, defined along \eqref{K1}-\eqref{K4}, solve
\eqref{C2}-\eqref{C1} in the distributional sense and satisfy \eqref{C3}-\eqref{C4} almost everywhere.  Using again \eqref{K1}-\eqref{K4}, one has that 
for any $i\in \mathcal I$
\[
s_i(x,0)=s_i^0 (x), \quad d_i(x,0)=d_i^0 (x),  \ \textrm{ for a.e. } x\in [0,1].
\]
We can therefore define for all $i\in \mathcal I$ and all $T>0$  a function $u_i\in H^1((0,1)\times (0,T))$ by 
\[
u_{i,t}=\frac{1}{2}(s_i+d_i)=:v_i, \ \  u_{i,x}=\frac{1}{2c_i}(s_i-d_i), \
\]
the constant of integration being chosen so that 
\[
u_i(x,t)=u_i^0(x)+\int_0^t v_i(x,s) ds\quad \textrm{ for a.e.} \  (x,t)\in (0,1)\times (0,T). 
\] 
Then $(u_i,v_i)\in C(\R ^+,H^1(0,1)\times L^2(0,1))$, and \eqref{AA5} follows from \eqref{C3}. We infer that 
$(u_i,v_i)_{i\in \mathcal I}$  is a (weak) solution of \eqref{A1}-\eqref{A5} and \eqref{A6c} 
which is continuous in time with values in $\mathcal H$.
Set $S(t)U^0= (u_i(t),v_i(t))_{i\in \mathcal I}$. Then it can be seen 
that $\big( S(t) \big)_{t\ge 0}$  is a strongly continuous semigroup in $\mathcal H$ whose generator is 
$A_T$. The proof of Theorem \ref{thm1} is complete.  
\end{proof}

\section{Finite-time extinction}
Pick any tree of depth $d\ge 1$, and define the sequence $(t_i)_{i\in \mathcal I}$ as follows
\ba
t_i&=& c_i^{-1} \qquad \textrm{ if } i\in {\mathcal N}_S \setminus \{ 0\}, \label{T1}\\
t_i&=& c_i^{-1}+\max_{j\in {\mathcal I}_i} t_j  \qquad \textrm{ if } i\in {\mathcal N}_M. \label{T2} 
\ea
Set $T(\mathcal R) =t_1$. Then it is easily seen that $T(\mathcal R)$ is the maximum of the quantities
\[
c_{i_1}^{-1}+c_{i_2}^{-1}+\cdots + c_{i_p}^{-1}, 
\]
where $p\ge 1$, $i_1=1$, $i_{q+1}\in I_{i_q}$ for $1\le q\le p-1$, and the final point of the edge of index $i_q$ is an external node (different from $\mathcal R$).
Define  $T(\mathcal T)$ as the largest of the $T(\mathcal R)$'s when the root $\mathcal R$ ranges over ${\mathcal N}_S$; that is, we take as root of the tree any 
external node, change the numbering of the edges and nodes, and  define the corresponding sequences $({\mathcal  I_i})_{i \in \mathcal I}$ and $(t_i)_{i\in \mathcal I}$.
Obviously, $T(\mathcal R)\le T(\mathcal T)\le 2T(\mathcal R)$.
\begin{Exam}
Consider again the tree drawn in Figure \ref{fig1}, and assume for simplicity that $c_i=1$ for all $i\in [1,11]$. Then $T({\mathcal R})=5$ and $T({\mathcal T})=7$. Indeed, if 
we take the node of index $n=12$ as (new) root, we obtain $T({\mathcal R}_{n=12})=7$. Similarly, we see that $T({\mathcal R}_{n=13})=7$, 
$T({\mathcal R}_{n=8})=T({\mathcal R}_{n=9})=6$, $T({\mathcal R}_{n=4})=5$, and $T({\mathcal R}_{n=10})=T({\mathcal R}_{n=11})=7$. 
\end{Exam}
\begin{thm}
\label{thm2}
Let $\mathcal T$ be a tree of root $\mathcal R$, and let $T(\mathcal R)$ and $T(\mathcal T)$ be as above.
Assume that the sequence $(\alpha _n)_{n\in {\mathcal N}_M}$ satisfies the condition
\be
\label{cond-FTS}
\alpha _n = k_n-2\qquad n\in {\mathcal N} _M. 
\ee 
Pick any initial data $U_0=\{ (u_{i}^0,u_{i} ^1 \}_{i\in \mathcal I}\in {\mathcal H}$.
\begin{enumerate}
\item[(i)] If $U_0\in {\mathcal H}_0$, then the solution $(u_i)_{i\in \mathcal I} $ of \eqref{A1}-\eqref{A5} and \eqref{A6a} satisfies
\be
\label{B1} 
u_i(.,t) \equiv 0, \qquad \forall t\ge 2T(\mathcal R), \ \forall i\in {\mathcal I};
\ee
\item[(ii)] The solution $(u_i)_{i\in \mathcal I} $ of \eqref{A1}-\eqref{A5} and \eqref{A6b} satisfies for some number $C\in \R$ 
\be
\label{B2} 
u_i(.,t) \equiv C,\qquad  \forall t\ge 2T(\mathcal R), \ \forall i\in {\mathcal I}. 
\ee 
\item[(iii)] The solution $(u_i)_{i\in \mathcal I} $ of \eqref{A1}-\eqref{A5} and \eqref{A6c} satisfies for some number $C\in \R$ 
\be
\label{B3} 
u_i(.,t) \equiv C, \qquad \forall t\ge T(\mathcal T), \ \forall i\in {\mathcal I}. 
\ee 
\end{enumerate} 
\end{thm} 
\begin{Rk} It is likely that the extinction time $T_e$ (i.e. the least time after which solutions remain constant) is given by   $2T(\mathcal R)$ in the cases (i) and (ii), and
$T(\mathcal T)$ in case (iii), so that the above results are sharp.
Actually, for one string, it is well known that $T_e=2/c_1$ for the solutions of \eqref{A1}-\eqref{A5} and \eqref{A6a}
(or  for the solutions of \eqref{A1}-\eqref{A5} and \eqref{A6b}), while $T_e=1/c_1$ for the solutions of \eqref{A1}-\eqref{A5} and \eqref{A6c}.
\end{Rk}
\begin{proof}
We use again the Riemann invariants $d_i,s_i$ defined in \eqref{Ri1}-\eqref{Ri2} that satisfy the transport equations \eqref{C2}-\eqref{C1}.  
We need the following
\begin{lem}
\label{lem2}
Let $\mathcal T$ be a tree, and let the sequence $(t_i)_{i\in\mathcal I}$ be as in \eqref{T1}-\eqref{T2}.
Assume that the sequence $(\alpha _n)_{n\in {\mathcal N}_M}$ satisfies \eqref{cond-FTS}. Then for any 
$U_0\in {\mathcal H}$ and any solution  $(u_i)_{i\in \mathcal I}$ of \eqref{A1}-\eqref{A5}, with corresponding Riemann invariants $d_i,s_i$, we have for all $i\in\cI$
\be
s_i(x,t) =0\quad \forall x\in [0,1],\ \forall t\ge t_i.\label{A21}
\ee 
\end{lem} 
\noindent
{\em Proof of Lemma \ref{lem2}.}  We argue by induction on the depth $d$ of the tree. If $d=1$, then there is only one edge ($\cI =\{1\}$)  and $s_1$ solves
\ba
s_{1,t}- c_1s_{1,x}&=& 0, \quad t>0,\ 0<x<1,  \label{A23}\\
s_1 (1,t) &=& 0, \quad t>0, \label{A24}\\ 
s_1(.,0) &=& s_1^0:=v_{1}^0+ c_1  u_{1,x}^0. \label{A25}
\ea 
Then it is easily seen that 
\be
s_{1}(x,t) =
\left\{
\begin{array}{ll}
s_1^0(x+c_1t)\quad &\textrm{if }x+c_1t\le 1,\\
0\quad & \textrm {if } x+c_1t\ge 1.
\end{array}
\right.
\ee
Thus
\[
s_1(x,t)=0\qquad \forall x\in[0,1],\ \forall t\ge c_1^{-1}
\]
and \eqref{A21} is established for $d=1$.

Assume now Lemma \ref{lem1} established for any tree of depth at most $d-1$, where $d\ge 2$. Pick a tree $\mathcal T$
of depth $d$, and a sequence $(\alpha_n)_{n\in {\mathcal N}_M}$ satisfying $(\mathcal P)$.  Denote by ${\mathcal R}'$ the node of index $n=1$, and 
 by ${\mathcal T}_i$, for $i=2,...,k_1$, the subtree of $\mathcal T$ of root ${\mathcal R}'$ and of first edge the edge of $\mathcal T$ of index $i$. 
 Since $\cT _i$ is of depth at most $d-1$, we infer from the induction hypothesis that for $i >1$
 \be
 s_i(x,t) =0 \qquad \forall x\in [0,1],\ \forall t\ge t_i. \  \label{A26}
 \ee 
It remains to prove \eqref{A21} for $i=1$.  Since the condition \eqref{cond-FTS} is satisfied for $n=1$, we infer from \eqref{cle} that
\[
s_1(1,t)=\sum_{i=2}^{k_1}s_i(0,t), \qquad \forall t\ge 0.
\]
It follows then from \eqref{A26} that 
\[
s_1(1,t)=0\qquad \forall t\ge \max _{i\in \mathcal I _1}t_i.
\]
Finally, using \eqref{A23}, we infer that 
\[
s_1(x,t)=0\qquad \forall x\in [0,1], \ \forall t\ge c_1^{-1} + \max _{i\in \mathcal I _1}t_i=t_1.
\]
The proof of Lemma \ref{lem2} is complete. \qed

Let us go back to the proof of Theorem \ref{thm2}.

(i) Assume first that $U_0 \in \cH _0$, and let 
$(u_i)_{i\in \cI}$  denote the solution of \eqref{A1}-\eqref{A5} and \eqref{A6a}.  From Lemma \ref{lem2}, we have that 
for all $i\in \mathcal I$ 
\be
\label{G1}
s_i(x,t) =0 \qquad \forall x\in [0,1],\ \forall t\ge T({\mathcal R}).
\ee
From \eqref{A6a}, we infer that $d_1(0,t)+s_1(0,t)=0$ for all $t\ge 0$, and hence
\[
d_1(0,t)=0,\qquad \forall t\ge T({\mathcal R}).
\]
Using \eqref{C2}, we infer that 
\[
d_1(x,t)=0,\qquad \forall x\in [0,1],\ \forall t\ge c_1^{-1}+ T({\mathcal R}).
\]
Combined with \eqref{S3}-\eqref{S3ter} (with $k=k_1$) and \eqref{G1}, this yields
\[
d_2(0,t)=\cdots =d_{k_1}(0,t)=0, \qquad \forall t\ge c_1^{-1} + \max_{i\in \mathcal I_1}c_i^{-1} +T(\mathcal R).
\] 
Using the second definition of $T(\mathcal R)$ and proceeding inductively, we arrive to 
\be
\label{G2}
d_i(x,t) =0 \qquad \forall i\in {\mathcal I}, \ \forall x\in [0,1],\ \forall t\ge 2T({\mathcal R}).
\ee
Gathering together \eqref{G1} and \eqref{G2}, we infer the existence of some constant $C\in\R$ such that
\[
u_i(x,t)=C,\qquad  \forall i\in {\mathcal I}, \ \forall x\in [0,1],\ \forall t\ge 2T({\mathcal R}).
\]
Using \eqref{A6a}, we see that $C=0$. This proves that solutions of  \eqref{A1}-\eqref{A5} and \eqref{A6a} are null for $t\ge T(\mathcal R)$. 
Combined with the strong continuity of the semigroup $(e^{tA_D})_{t\ge 0}$ in ${\mathcal H}_0$, this yields the finite time stability. 

(ii) Assume now that $u_0\in \cH $ and let $(u_i)_{i\in \cI}$ denote the solution of \eqref{A1}-\eqref{A5} and \eqref{A6b}. 
From \eqref{A6a}, we infer that $d_1(0,t)-s_1(0,t)=0$ for all $t\ge 0$. The same proof as in (i) then yields
\[
s_i(x,t)=d_i(x,t)=0, \qquad \forall i\in {\mathcal I}, \ \forall x\in [0,1],\ \forall t\ge 2T({\mathcal R}).
\]
Thus there exists a constant $C\in \R$ such that 
\[
u_i(x,t)=C,\qquad  \forall i\in {\mathcal I}, \ \forall x\in [0,1],\ \forall t\ge 2T({\mathcal R}).
\]

(iii) Pick a solution $(u_i)_{i\in \cI}$  of \eqref{A1}-\eqref{A5} and \eqref{A6c}. Then it follows from Lemma \ref{lem2} that 
for all $i\in \mathcal I$ 
\be
\label{G3}
s_i(x,t) =0 \qquad \forall x\in [0,1],\ \forall t\ge T({\mathcal R}).
\ee
For any given $i\in \cI$, we pick a sequence $i_1<i_2<\cdots < i_p$ such that $i_1=1$, $i=i_q$ for some $q\in [1,p]$, and the final point of the edge of index $i_p$ is
an external point, that we call $\tilde{\mathcal R}$. If we exchange $\mathcal R$ and $\tilde{\mathcal R}$, we notice that $d_i$ is linked to the $\tilde s_j$'s (associated
with the new root $\tilde{\mathcal R}$) by:
\[
d_i(x,t)=\tilde s_{i_p-i+1}(1-x,t).
\] 
We infer that 
\be
\label{G4}
d_i(x,t) =0 \qquad \forall x\in [0,1],\ \forall t\ge T({\mathcal T}).
\ee
Therefore, there exists a constant $C\in \R$ such that 
\[
u_i(x,t)=C,\qquad  \forall i\in {\mathcal I}, \ \forall x\in [0,1],\ \forall t\ge T({\mathcal T}).
\]
The proof of Theorem \ref{thm2} is complete.
\end{proof}
\section{Sharpness of the condition \eqref{cond-FTS}}
The condition \eqref{cond-FTS}, which is sufficient to yield the finite-time stability, is expected to be also necessary. A way to prove it is to search for
an eigenvalue of the underlying operator. Indeed, if we can find an eigenvalue,  then the corresponding exponential solution will not steer 0 in  finite time.
This program can be achieved  when the geometry is sufficiently simple, namely when $d=2,3$. Actually, we will  consider any value of  the sequence of 
coefficients  $(\alpha _n)_{n\in {\mathcal N}_M}$, and exhibit an eigenvalue
of the underlying operator when \eqref{cond-WP} holds and \eqref{cond-FTS} fails.
We shall consider 
\begin{enumerate}
\item a star-shaped tree, with the homogeneous Dirichlet boundary condition at one external node and the transparent boundary conditions at the other external nodes; 
\item a tree with two internal nodes, for which a transparent boundary condition is applied at each external node. 
\end{enumerate}
\subsection{The star-shaped tree}
Assume that $\cT$ is a star-shaped tree with $N$ edges ($d=2$, $k_1=N$), and consider the boundary conditions \eqref{A3}-\eqref{A5} and \eqref{A6a}.
(See figure \ref{fig2}.)
\begin{figure}[http]
\begin{center}
\includegraphics[scale=0.5]{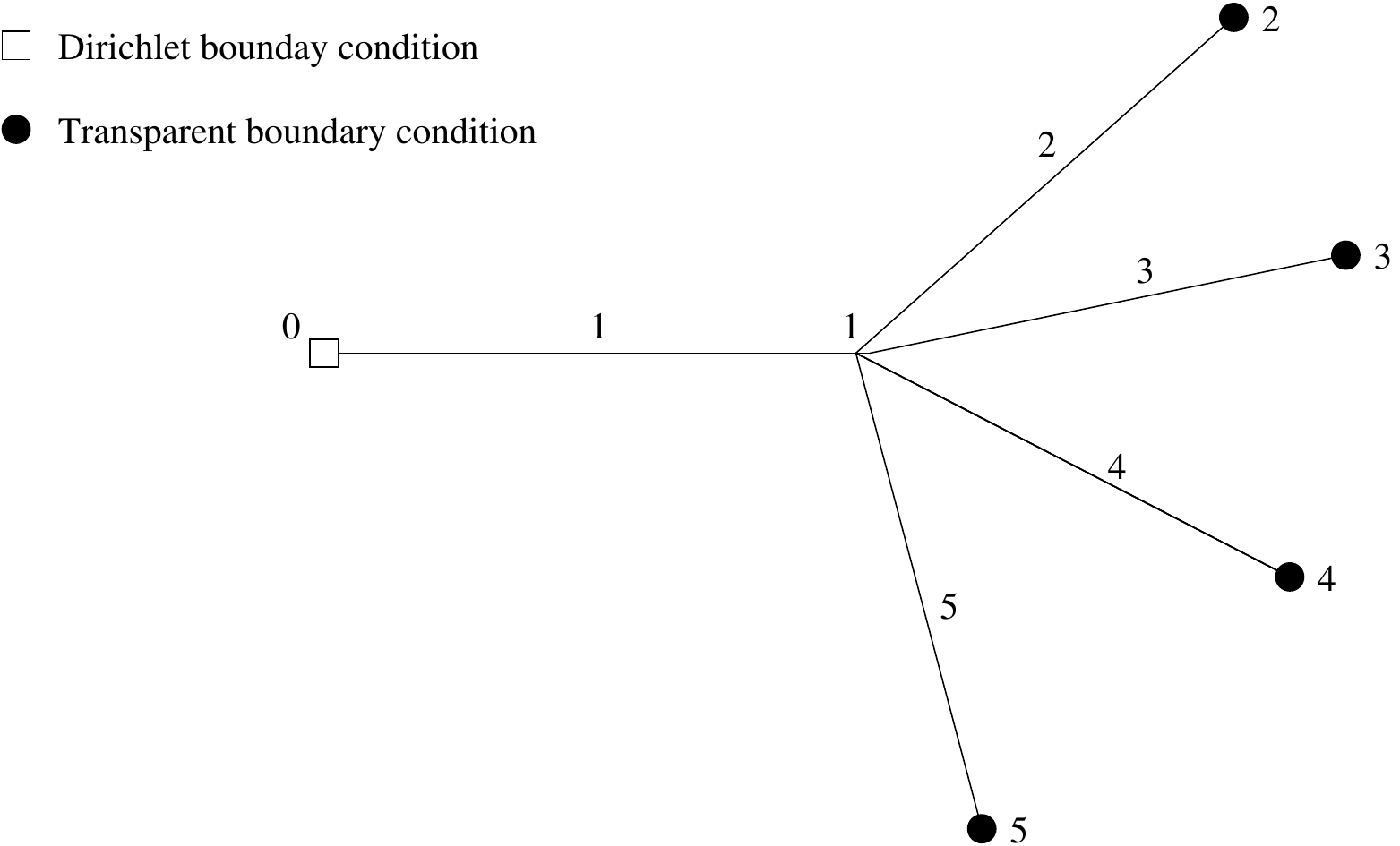}
\end{center}
\caption{A star-shaped tree.}
\label{fig2}
\end{figure}
 
We assume that $\alpha _1\ne N$, so that the system   \eqref{A1}-\eqref{A5} and \eqref{A6a} is well-posed in ${\mathcal H}_0$ according to Theorem \ref{thm1}. 
According to Theorem \ref{thm2}, there is a finite-time stabilization when $\alpha _1=N-2$. We shall show that this condition is sharp, i.e. that
a finite-time stabilization cannot hold if $\alpha _1\not\in \{N-2,N\}$.

Let $\alpha _1\in \R$ be given. The operator $A_D$ reads
\[
A_D\big( (u_i,v_i)_{i\in \cI}\big) = (v_i,c_i^2u_i'')_{i\in \cI}
\] 
with
\begin{multline*}
D(A_D) = \{ (u_i,v_i)_{i\in\cI} \in \cH _0;\ (v_i,c_i^2u_i'')_{i\in \cI} \in \cH _0, \ c_iu_i'(1)=-v_i(1)\text{ for } 2\le i\le N\ \\
\sum_{2\le i\le N} c_iu_i'(0) - c_1u_1'(1) = -\alpha _1 v_1(1), \ \text{ and } \  (u_i(0),v_i(0))=(u_1(1),v_1(1)) \ \text{ for } 2\le i\le  N\},
\end{multline*}
where $' =d/dx,\  ''=d^2/dx^2,$ etc.
Setting $U:=(u_i,v_i)_{i\in \cI}$, we see that \eqref{A1}-\eqref{A5} and \eqref{A6a} may be written as
\ba
U_t&=&A_DU \label{W1}\\
U(0)&=&U_0=(u_i^0,u_i^1)_{i\in \cI} \label{W2}
\ea
If $A_DU_0=\lambda U_0$ with $U_0\ne 0$, then the solution 
$U$ of \eqref{W1}-\eqref{W2} reads $U(t)=e^{\lambda t} U_0$  (exponential solution), and hence $||U(t)|| _\cH =e^{(\text{Re}\,  \lambda )t } ||U_0||_\cH >0$ for all
$t\ge 0$.  Thus if the operator $A_D$ has {\em at least} one eigenvalue, then the finite-time stabilization cannot hold.   
\begin{prop}
\label{prop1}
Let $\cT$ denote a star-shaped tree with $N$ edges, and assume that $\alpha _1\ne N$. Then 
the operator $A_D$ has at least one eigenvalue if, and only if, 
\be
\label{F1}
\alpha _1\ne N-2.
\ee
Furthermore, if \eqref{F1} holds, then the discrete spectrum of $A_D$ is $\sigma _d(A_D)=\{\lambda _k;\ k\in \Z\}$ where
\be
\label{eigenvalue}
\lambda _k = \frac{c_1}{2} \log_{-\frac{\pi}{2}} \frac{N-2-\alpha_1}{N-\alpha _1}  + ic_1k\pi
\ee 
and $\log_{     -\frac{\pi}{2}   }$ denotes the usual determination of the logarithm in $\C \setminus i\R ^-$. In particular, if \eqref{F1} holds, then 
the finite-time stabilization of  \eqref{A1}-\eqref{A5} and \eqref{A6a} in ${\mathcal H}_0$ fails. 
\end{prop}  
\begin{Rk} \label{Rk2}
1.Note that 
\[
\log_{     -\frac{\pi}{2}   } (z) =\left\{ 
\begin{array}{ll}
\log |z| \quad&\textrm{ if } z\in (0,+\infty), \\
\log |z| + i\pi \quad &\textrm{ if } z\in (-\infty , 0).
\end{array} 
\right.
\]
2. If we replace the Dirichlet boundary condition $u_1(0,t)=0$ by the transparent boundary condition
$u_{1,t}(0,t)= c_1u_{1,x}(0,t)$ and take {\em any} value $\alpha_1\ne N$, then since $d_1(0,t)=s_2(1,t)=\cdots =s_N(1,t)=0$ for all $t\ge 0$, 
we infer from \eqref{S3}-\eqref{S3ter} and \eqref{S4} that  $s_1(1,t)=d_2(0,t)=\cdots =d_N(0,t)=0$ for all $t\ge \max_{1\le i\le N}c_i^{-1}$, so that for some 
constant $C\in R$
\[
u_i(x,t)=C,\qquad \forall i\in [1,N],\ \forall x\in [0,1],\ \forall t\ge 2\max_{1\le i\le N}c_i^{-1}.
\]
\end{Rk}
\begin{proof}
Let $\lambda \in \C$ and $U=(u_i,v_i)_{i\in \cI} \in D(A_D)$. Then the equation $A_DU=\lambda U$ is equivalent to the following system
\ba
&&(v_i,c_i^2u_i'')=\lambda (u_i,v_i), \qquad 1\le i\le N,  \label{Z1}\\
&&u_1(0)=0, \label{Z2} \\
&&c_iu_i'(1) = -v_i(1),\qquad  2\le i\le N, \label{Z3} \\
&&\sum_{2\le i \le N} c_iu_i'(0) - c_1 u_1'(1) = -\alpha _1 v_1(1), \label{Z4}\\
&&u_i(0)=u_1(1), \qquad 2\le i\le N. \label{Z5}
\ea
Note that the conditions $v_1(0)=0$ and $v_i(0)=v_1(1) $ for $2\le i\le N$ are satisfied whenever \eqref{Z1}-\eqref{Z2} and \eqref{Z5} hold.
\eqref{Z1} is easily solved as
\be
u_i(x) = a_i e^{\lambda x/c_i} + b_i e^{-\lambda x/c_i},\ v_i(x)=\lambda u_i(x), \  1\le i\le N, \label{Z6} 
\ee
where $a_i,b_i\in \C$ are constants to be determined. Substituting the above expression of $u_i(x)$ in \eqref{Z2}-\eqref{Z5}  yields the system
\ba
&&a_1+b_1=0, \label{Z7} \\
&&\lambda a_i=0, \qquad 2\le i\le N, \label{Z8}\\
&&\lambda \sum_{2\le i\le N} (a_i-b_i) -\lambda (a_1 e^{\lambda /c_1} -b_1 e^{-\lambda /c_1} ) = -\alpha _1 \lambda
(a_1 e^{\lambda /c_1} + b_1 e^{-\lambda /c_1} ), \qquad \label{Z9}\\
&&a_i+b_i = a_1 e^{\lambda /c_1} +b_1 e^{-\lambda /c_1}, \qquad 2\le i\le N. \label{Z10} 
\ea
If $\lambda =0$, we infer from \eqref{Z6}-\eqref{Z7} and \eqref{Z10} that  $U=0$, which is excluded. Assume from now on  that $\lambda \ne 0$. Then the system 
\eqref{Z7}-\eqref{Z10} is found to be equivalent to the system 
\ba
&&b_1=-a_1, \label{Z11} \\
&& a_i=0, \qquad 2\le i\le N, \label{Z12}\\
&&-(N-1) a_1 (e^{\lambda /c_1}  -e^{-\lambda /c_1} ) 
-a_1 (e^{\lambda /c_1}  +e^{-\lambda /c_1} ) = -\alpha _1 
a_1 (e^{\lambda /c_1} - e^{-\lambda /c_1} ),\qquad  \label{Z13}\\
&&b_i = a_1 (e^{\lambda /c_1} -e^{-\lambda /c_1}), \qquad 2\le i\le N. \label{Z14} 
\ea
The existence of a nontrivial solution ($a_1\ne 0$) holds if, and only if, the coefficient above $a_1$ in \eqref{Z13} vanishes, i.e.
\be
(-N+\alpha _1) e^{\lambda /c_1} + (N-2 -\alpha _1 ) e^{-\lambda /c_1} =0. \label{Z20}
\ee
For $\alpha _1\ne N$, \eqref{Z20} is equivalent to
\[
e^{\frac{2\lambda}{c_1}} =\frac{N-2-\alpha _1}{N-\alpha _1} \cdot \label{Z3000}
\] 
\eqref{Z3000} has a solution $\lambda \in \C$ if and only if $\alpha _1\ne N-2$, and in that case the solutions of \eqref{Z3000} read 
\be
\label{ZZZ}
\lambda _k = \frac{c_1}{2} \log_{-\frac{\pi}{2}} \frac{N-2-\alpha_1}{N-\alpha _1} + ic_1k\pi, \quad k\in \Z.
\ee
\end{proof}
\begin{Rk} \label{Rk20}
For $k\in \Z$ and $\lambda _k$ as in \eqref{ZZZ}, we introduce the sequence of eigenfunctions $U_k=((u_{i,k},v_{i,k}))_{ 1 \le i\le N, k\in \Z}$ where 
\begin{eqnarray*}
u_{1,k} (x)  &=& e^{\lambda _k x/c_1} - e^{-\lambda _kx/c_1}, \quad v_{1,k} (x) =\lambda _k u_{1,k}(x), \\  
u_{i,k} (x)  &=& (e^{\lambda _k /c_1} - e^{-\lambda _k/c_1}) e^{-\lambda _k x/c_i} , \quad v_{i,k} (x) =\lambda _k u_{i,k}(x), \quad\text{ for }  2\le i\le N.
\end{eqnarray*}
Then the family $ (a_kU_k )_{k\in \Z}$ may fail to be a Riesz basis in $\mathcal H_0$ for any choice of the sequence of numbers  $(a_k)_{k\in \Z}$. Consider e.g.
$N=2$ and $c_2 = c_1/2$.  Then, for $N-2<\alpha _1 <N$, 
$$  u_{2,k} (x)  = (e^{\lambda _k /c_1} - e^{-\lambda _k/c_1}) e^{ - \log \vert \frac{N-2-\alpha _1}{N-\alpha _1} \vert x-i\pi x } e^{- i 2 k \pi x }. $$
Let $U=(u_i,v_i)_{i=1,2}\in {\mathcal H_0}$ be given. If $( a_kU_k )_{k\in \Z}$ is a Riesz basis in $\mathcal H_0$, then 
$U$ can be expended in terms of the $U_k$'s in $\mathcal H_0$ as
\[
(u_i,v_i)=\sum_{k\in \Z} d_ka_k (u_{i,k},v_{i,k}), \quad i=1,2
\]
for some sequence $(d_k)_{k\in \Z} \in L^2(\Z )$. Writing 
\[
e^{  \log \vert \frac{N-2-\alpha _1}{N-\alpha _1} \vert x  + i\pi x }u_2(x) =  \sum_{k\in \Z} c_k e^{-i2k\pi x}
\] 
we have, by harmonicity, that 
\[
c_k=d_ka_k (e^{\lambda _k/c_1} -e^{-\lambda _k/c_1} ), \quad k\in \Z ,
\]
and hence 
\[
u_1(x) = \sum_{k\in \Z} \frac{c_k}{e^{\lambda _k/c_1} - e^{-\lambda _k/c_1}} (e^{\lambda _k x/c_1} - e^{ - \lambda _k x/c_1}) 
\]
in $L^2(0,1)$. Therefore, $u_1$ is uniquely determined by the $c_k$'s, and hence by $u_2$, which is a 
property much stronger than the conditions $u_1(0)=0$ and $u_1(1)=u_2(0)$ present
in the definition of $\mathcal H_0$. This shows that the family $(a_kU_k)_{k\in \Z}$ is not  total in $\mathcal H_0$.  

It is natural to conjecture a decay of all the trajectories like 
\be
||U(t)||_{\mathcal H_0} \le C (\alpha _1) e^{\frac{c_1}{2} \log \left\vert \frac{N-2-\alpha _1}{N-\alpha _1}\right\vert t} ||U(0) ||_{\mathcal H_0}, \qquad t\ge 0, 
\label{decay}
\ee
for $N-2<\alpha _1<N$. (Note that $\lim_{\alpha _1 \searrow N-2} \log \vert \frac{N-2-\alpha _1}{N-\alpha _1} \vert  = - \infty $.) Without a Riesz basis of eigenvectors
 in the full space $\mathcal H_0$,
the validity of \eqref{decay} seems hard to check. 
\end{Rk}

\subsection{The tree with two internal nodes}
We   assume now that $\cT$ is a tree with $N+1$ nodes, two of which being multiple  ($d=3$, $\cN _M = \{ 1, 2\}$, $k_1\ge 2$, $k_2\ge 2$, $k_1+k_2=N+1$),
and we consider the boundary conditions \eqref{A3}-\eqref{A5} and \eqref{A6c}. (See Figure \ref{fig3}.)
\begin{figure}[http]
\begin{center}
\includegraphics[scale=0.5]{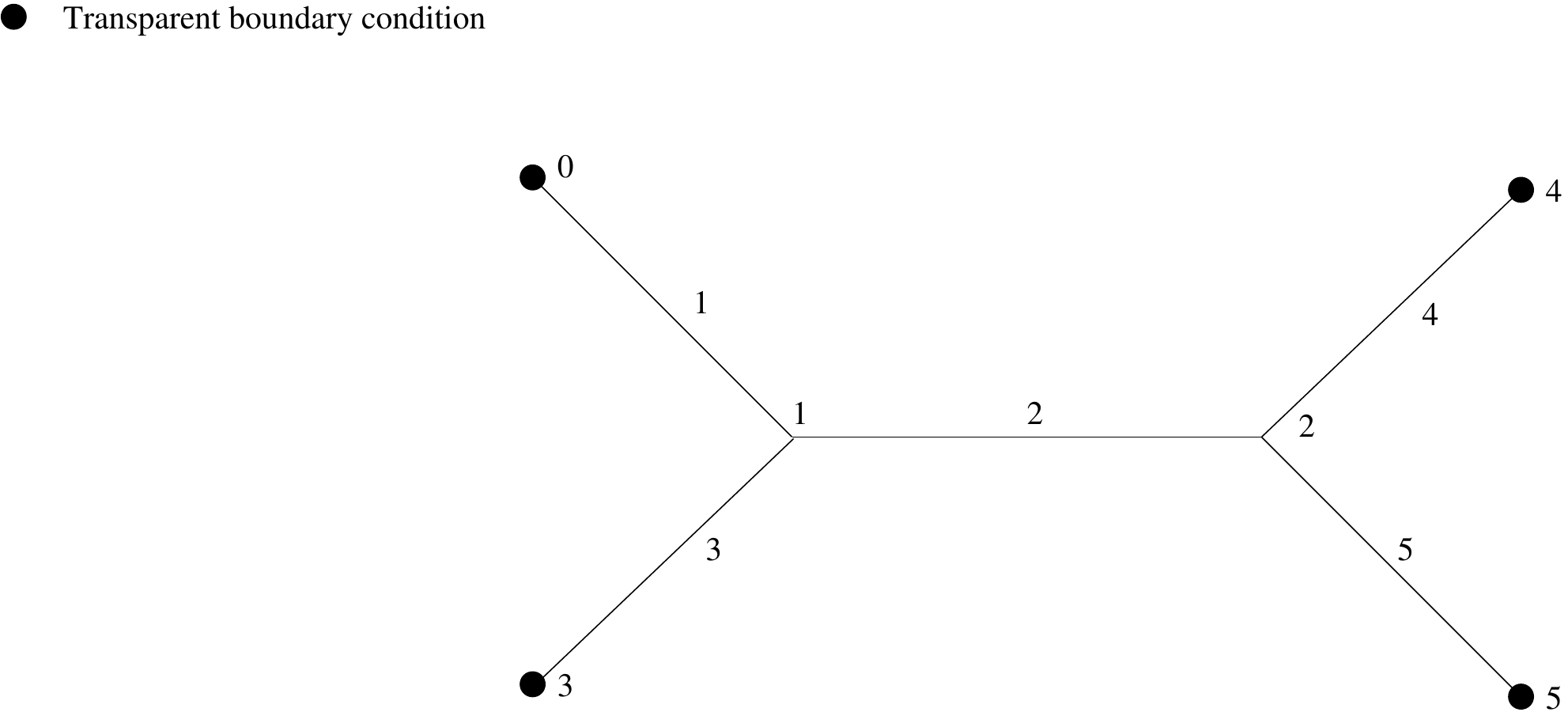}
\end{center}
\caption{A bone-shaped tree.}
\label{fig3}
\end{figure}
We will let $\alpha_1$ and $\alpha_2$ range over $\R$, assuming only that 
\eqref{cond-WP} holds. In particular,
when $\alpha_1=\alpha_2=0$, there is no damping at the internal nodes $n=1,2$. We shall show that the finite-time stabilization cannot hold in that case, 
because of the (partial but continuous) bounces of waves at the internal nodes. Note that for this geometry, condition \eqref{cond-FTS} reads 
\be
\label{AA1}
\alpha _1 =k_1-2,\quad \alpha _2=k_2-2.
\ee
Here, we shall show that there is an eigenvalue  (so that the finite-time stability fails) if, and only if,  both $\alpha _1 \ne k_1-2$ and $\alpha _2\ne k_2-2$. Notice
that this condition is stronger than $(\alpha _1,\alpha _2)\ne (k_1-2, k_2-2)$. We shall prove that, when
\be
(\alpha _1,\alpha_2)\in \{ k_1-2 \} \times (\R \setminus \{ k_2 \}) \cup  (\R \setminus \{ k_1 \}) \times  \{ k_2-2 \} , \label{cond-FTS-new}
\ee
then the finite-time stability (to constant functions) occurs. We conclude that, when $d=3$ and transparent boundary conditions
are imposed at all the external nodes, a necessary and sufficient condition for the finite-stability (to constant functions) is \eqref{cond-FTS-new}. 
The interpretation is that the nodes satisfying \eqref{cond-FTS} and for which all the adjacent nodes but one are external, are ``transparent'' and can be  ``removed''
from the tree.

Let $(\alpha _1,\alpha _2)\in \R ^2$ be given. The operator $A_T$ reads then
\[
A_T\big( (u_i,v_i)_{i\in \cI}\big) = (v_i,c_i^2u_i'')_{i\in \cI}
\] 
with domain
\begin{eqnarray*}
&&D(A_T) = \{ (u_i,v_i)_{i\in\cI} \in \cH ;\ (u_i,v_i)_{i\in \cI} \in \prod_{i\in \cI}[H^2(0,1)\times H^1(0,1)] ,\\
&&\qquad \quad  c_1u_1'(0) =v_1(0),  \   c_iu_i'(1)=-v_i(1)\text{ for } i\in \{ 3,...,N \}  \\
&&\qquad\quad  \sum_{2\le i\le k_1} c_iu_i'(0) - c_1u_1'(1) = -\alpha _1 v_1(1), \!\!\!\! \sum_{k_1+1 \le i\le N} c_iu_i'(0) - c_2u_2'(1) = -\alpha _2 v_2(1),  \\
&&\qquad \quad  (u_i(0),v_i(0))=(u_1(1),v_1(1)) \ \text{ for } 2\le i\le  k_1, \\
&&\qquad \quad  (u_i(0),v_i(0))=(u_2(1),v_2(1)) \ \text{ for } k_1+1\le i\le  N\}.
\end{eqnarray*}
Setting $U=(u_i,v_i)_{i\in \cI}$, we see that \eqref{A1}-\eqref{A5} and \eqref{A6c} may be written as
\ba
U_t&=&A_TU, \label{WW1}\\
U(0)&=&U_0=(u_{i}^0,u_{i}^1)_{i\in \cI}. \label{WW2}
\ea
\begin{prop}
\label{prop2}
Let $\cT$ denote a tree with $N$ edges and two internal nodes ($\cN _M=\{1,2\}$), and assume that 
\be
\alpha _1\ne k_1  \textrm{ and }   \alpha_2 \ne k_2.\label{LL1}
\ee
Then  the operator $A_T$ has at least one eigenvalue if, and only if, 
\be
\alpha _1\ne k_1-2  \textrm{ and }    \alpha _2\ne k_2-2. \label{LL2}
\ee
Furthermore, if \eqref{LL2} holds, then the discrete spectrum of $A_T$ is $\sigma _d(A_T)=\{\lambda _k;\ k\in \Z\}$ where
\be
\label{eigenvalue2}
\lambda _k = \frac{c_2}{2} \log_{-\frac{\pi}{2}} \frac{(2+\alpha_1-k_1)(2+\alpha_2-k_2)}{(\alpha _1-k_1)(\alpha_2 -k_2)}  +ic_2k\pi .
\ee 
In particular, the finite-time stability to constant functions does not hold for \eqref{A1}-\eqref{A5} and \eqref{A6c}. 
Finally, if \eqref{cond-FTS-new} is satisfied, then the finite-time stability to constant functions holds. 
\end{prop}  
\begin{proof} First, $A_T$ generates a strongly continuous semigroup of operators in $\mathcal H$ by \eqref{LL1} and Theorem \ref{thm1}. 
Let $\lambda \in \C$ and $U=(u_i,v_i)_{i\in \cI} \in D(A_T)$. Then the equation $A_TU=\lambda U$ is equivalent to the following system
\ba
&&(v_i,c_i^2u_i'')=\lambda (u_i,v_i) \label{ZZ1}\\
&&c_1u_1'(0)=v_1(0) \label{ZZ2} \\
&&c_iu_i'(1) = -v_i(1),\qquad  3\le i\le N \label{ZZ3} \\
&&\sum_{2\le i \le k_1} c_iu_i'(0) - c_1 u_1'(1) = -\alpha _1v_1(1) \label{ZZ4}\\
&&\sum_{k_1+1\le i \le N} c_iu_i'(0) - c_2 u_2'(1) = -\alpha _2v_2(1) \label{ZZ5}\\
&&u_i(0)=u_1(1), \qquad 2\le i\le k_1, \label{ZZ6}\\
&&u_i(0)=u_2(1), \qquad k_1+1\le i\le N. \label{ZZ7}
\ea
Note that the conditions $v_i(0)=v_1(1)$ for $2\le i\le k_1$  and $v_i(0)=v_2(1) $ for $k_1 + 1 \le i\le N$ are satisfied whenever \eqref{ZZ1} and \eqref{ZZ6}-\eqref{ZZ7} hold.
\eqref{ZZ1} is easily solved as
\be
u_i(x) = a_i e^{\lambda x/c_i} + b_i e^{-\lambda x/c_i},\ v_i=\lambda u_i, \ i\in \cI , \label{ZZZ6} 
\ee
where $a_i,b_i\in \C$ are constants to be determined. Substituting the above expression of $u_i(x)$ in \eqref{ZZ2}-\eqref{ZZ7}  yields the system
\ba
&&\lambda b_1=0, \label{ZZ10} \\
&&\lambda a_i=0, \qquad 3\le i\le N, \label{ZZ11}\\
&&\lambda \sum_{2\le i\le k_1 } (a_i-b_i) -\lambda (a_1 e^{\lambda /c_1} -b_1 e^{-\lambda /c_1} ) = -\alpha _1 \lambda
(a_1 e^{\lambda /c_1} + b_1 e^{-\lambda /c_1} ), \qquad \label{ZZ12}\\
&&\lambda \sum_{k_1+1 \le i\le N } (a_i-b_i) -\lambda (a_2 e^{\lambda /c_2} -b_2 e^{-\lambda /c_2} ) = -\alpha _2 \lambda
(a_2 e^{\lambda /c_2} + b_2 e^{-\lambda /c_2} ), \qquad \label{ZZ13}\\
&&a_i+b_i = a_1 e^{\lambda /c_1} +b_1 e^{-\lambda /c_1}, \qquad 2\le i\le k_1, \label{ZZ14} \\
&&a_i+b_i = a_2 e^{\lambda /c_2} +b_2 e^{-\lambda /c_2}, \qquad k_1+1\le i\le N. \label{ZZ15} 
\ea
If $\lambda =0$, we infer from \eqref{ZZ14}-\eqref{ZZ15} and \eqref{ZZZ6} that $u_i(x)=a_1+b_1$ for all $i\in \cI$, i.e.  
$U=const$, which is excluded. Assume from now on that $\lambda \ne 0$. Then
\eqref{ZZ10}-\eqref{ZZ15} is equivalent to the system 
\ba
&&b_1=0, \label{ZZ20} \\
&& a_i=0, \qquad 3\le i\le N, \label{ZZ21}\\
&& b_2=a_1e^{\lambda /c_1} - a_2,\label{ZZ22}\\
&&b_i= a_1e^{\lambda /c_1}, \qquad 3\le i\le k_1, \label{ZZ23}\\
&&b_i= a_2e^{\lambda /c_2} + b_2 e^{-\lambda /c_2} , \qquad k_1+1\le i\le N, \label{ZZ24}\\
&& 2a_2 + (\alpha _1 -k_1) e^{\lambda /c_1} a_1=0 ,\qquad  \label{ZZ25}\\
&& [(-N+k_1-1+\alpha _2 ) e^{\lambda /c_2} +(N-k_1-1-\alpha _2)e^{-\lambda /c_2} ] a_2\nonumber\\
&& \qquad  +
(-N+k_1+1+\alpha _2 ) e^{-\lambda /c_2} e^{\lambda /c_1} a_1=0.\qquad \label{ZZ26} 
\ea
The existence of a nontrivial solution ($(a_1,a_2)\ne (0,0)$) holds if, and only if, the 
determinant of the system \eqref{ZZ25}-\eqref{ZZ26} in $e^{\lambda /c_1}a_1$ and $a_2$ vanishes, i.e.
\[
(2+\alpha _1 -k_1)(-N+k_1+1+\alpha _2)e^{-\lambda /c_2} -(\alpha _1 -k_1) (-N+k_1-1+\alpha _2) e^{\lambda /c_2}=0.
\]
Since $-N+k_1=1-k_2$, this can be expressed as 
\[
(2+\alpha _1 -k_1)(2+\alpha _2-k_2)e^{-\lambda /c_2} -(\alpha _1 -k_1) (\alpha _2-k_2) e^{\lambda /c_2}=0.
\]
Using \eqref{LL1}, the last equation is equivalent to 
\be
e^{\frac{2\lambda}{c_2}} = \frac{(2+\alpha_1-k_1)(2+\alpha_2-k_2)}{(\alpha _1-k_1)(\alpha_2 -k_2)} 
 \cdot \label{Z6000}
\ee
\eqref{Z6000} has a solution $\lambda \in \C$ if and only if $(2+\alpha_1-k_1)(2+\alpha_2-k_2) \ne 0$, and in that case the solutions of \eqref{Z6000} read 
\[
\lambda _k = \frac{c_2}{2} \log_{-\frac{\pi}{2}}  \frac{(2+\alpha_1-k_1)(2+\alpha_2-k_2)}{(\alpha _1-k_1)(\alpha_2 -k_2)}  + ic_2k\pi, \quad k\in \Z.
\]
Assume finally that \eqref{cond-FTS-new} holds, e.g. $\alpha _1=k_1-2$ and $\alpha _2\in \R \setminus \{ k_2 \}$. Since transparent boundary conditions
are applied at all the external nodes, we have
\begin{eqnarray*}
s_i(1,t)&=&0, \qquad i=3,...,N, \ t\ge 0,\\
d_1(0,t)&=&0, \qquad t\ge 0.
\end{eqnarray*}
This implies
\ba
s_i(x,t)&=&0, \qquad i=3,...,N, \ x\in [0,1],\ t\ge c_i^{-1},\label{H1}\\
d_1(x,t)&=&0, \qquad x\in [0,1],\ t\ge c_1^{-1}.\label{H2}
\ea
It follows from \eqref{cle} and \eqref{H1} that 
\[
s_2(0,t)=s_1(1,t),\qquad t\ge \max_{3\le i\le k_1} c_i^{-1}.
\]
Combined with the continuity condition $u_1(1,t)=u_2(0,t)$, this yields
\[
d_2(0,t)=d_1(1,t)=0 \qquad t\ge \max_{i\in \{1\} \cup [3,k_1] }c_i^{-1}.
\]
The same argument as in Remark \ref{Rk2} shows that $s_2(1,t)=d_{k_1+1}(0,t)=\cdots =d_N(0,t)=0$ for $t$ large enough.
This in turn implies  $s_1(1,t)=d_3(0,t)=\cdots =d_{k_1}(0,t)=0$ for $t$ large enough. We conclude that for some constants $C$ and $T$, 
$u_i(x,t)=C$ for all $i\in\mathcal I$, all $x\in [0,1]$ and all $t\ge T$.   
\end{proof}


\begin{thebibliography}{9}
\addcontentsline{toc}{section}{References}
\bibitem{AJ} K. Ammari, M. Jellouli, {\em Stabilization of star-shaped networks of strings}, Differential and Integral Equations {\bf 17} (2004), No. 11-12, 1395--1410.  
\bibitem{AJK} K. Ammari, M. Jellouli, M. Khenissi, {\em Stabilization of generic trees of string}, J. Dyn. Control Syst., {\bf 11}  (2005),  177--193.
\bibitem{BR} A. Bacciotti, L. Rosier, Liapunov functions and stability in control theory. Second edition. Communications and Control Engineering Series. 
Springer-Verlag, Berlin, 2005.
\bibitem{BB} S.P. Bhat,  D.S. Berstein, {\em Finite-time stability of continuous autonomous systems}, SIAM J. Control Optim.
38 (2000), no. 3, 751--766.  
\bibitem{CZ} S. Cox,  E.  Zuazua, {\em The rate at which energy decays in a string damped at one end}, Indiana Univ. Math. J. {\bf 44} (1995), no. 2, 545--573. 
\bibitem{DZcras} R. D\'ager, E. Zuazua,  {\em Controllability of tree-shaped networks of vibrating strings},  C. R. Acad. Sci. Paris S\'er. I Math. {\bf 332} (2001), no. 12, 1087--1092.
\bibitem{DZbook} R. D\'ager, E.  Zuazua, Wave propagation, observation and control in 1-d flexible multi-structures. Math\'ematiques $\&$ Applications (Berlin) 
[Mathematics $\&$ Applications], 50. Springer-Verlag, Berlin, 2006.
\bibitem{gugat} M. Gugat, {\em Optimal boundary feedback stabilization of a string with moving boundary}, IMA J. Math. Control Inform., {\bf 25}  (2008), 111--121.
\bibitem{GDL}  M. Gugat, M. Dick, G. Leugering, {\em Gas flow in fan-shaped networks: classical solutions and feedback stabilization},
 SIAM J. Control Optim. {\bf 49} (2011), no. 5, 2101--2117.
 \bibitem{GS} M. Gugat, M. Sigalotti, {\em Stars of vibrating strings: switching boundary feedback stabilization}, Netw. Heterog. Media {\bf 5} (2010), no. 2, 299--314.
 \bibitem{haimo} V. T Haimo, {\em Finite time controllers}, SIAM J. Control Optim. {\bf 24} (1986), no. 4, 760--770. 
 \bibitem{komornik} V. Komornik,  Exact Controllability and Stabilization. The Multiplier Method. Research in Applied Mathematics, Masson, Wiley, 1994.
 \bibitem {LLS} J. E. Lagnese, G. Leugering,  E. J. P. G. Schmidt, Modeling, Analysis and Control of Dynamic Elastic Multi-link Structures, Systems $\&$ Control: 
 Foundations $\&$ Applications, Birkh\"auser Boston, Boston, MA, 1994.
\bibitem{majda} A. Majda, {\em Disappearing Solutions for the Dissipative Wave Equation}, Indiana University Mathematics Journal {\bf 24} (1975), No. 12, 1119--1133. 
\bibitem{MP} E. Moulay, W. Perruquetti, {\em Finite time stability and stabilization: state of art}, Advances in variable structure and sliding mode control, 23--41, Lecture
Notes in Control and Inform. Sci., 334, pringer, Berlin, 2006. 
\bibitem{PR-IFAC} V. Perrollaz, L. Rosier, {\em Finite-time stabilization of hyperbolic systems over a bounded interval}, in Proceedings
of 1st IFAC Workshop on Control of Systems Governed by Partial Differential Equations, Paris, 2013.
\bibitem{PR} V. Perrollaz, L. Rosier, {\em Finite-time stabilization of $2\times 2$ hyperbolic systems on tree-shaped networks}, SIAM J. Control Optim. {\bf 52} (2014), no. 1, 
143--163.  
\bibitem{VZ}  J. Valein, E. Zuazua, {\em Stabilization of the wave equation on 1-D networks}, SIAM J. Control Optim. {\bf 48} (2009),  no. 4, 2771--2797.
\bibitem{ZX} Y. Zhang,  G. Xu, {\em  Controller design for bush-type 1-D wave networks}, ESAIM Control Optim. Calc. Var. {\bf 18} (2012), no. 1, 208--228. 
 \bibitem{XLL} G. Q. Xu,  D. Y. Liu, and Y. Q. Liu, {\em Abstract second order hyperbolic system and applications to controlled network of strings}, 
 SIAM J. Control Optim. {\bf 47}  (2008), No. 4, 1762--1784. 
 \end{thebibliography}
\end{document}